\documentclass[11pt,leqno]{article}
\usepackage{amssymb,amsfonts,amsmath,amsthm,amscd,mathrsfs}
\usepackage{psfrag}
\def\IC{{\mathbb C}}

\def\IF{{\mathbb F}}

\def\IR{{\mathbb R}}
\def\IN{{\mathbb N}}
\def\IT{{\mathbb T}}
\def\IQ{{\mathbb Q}}
\def\IZ{{\mathbb Z}}

\def\IT{{\mathbb T}}

\def\n{\noindent}
\def\dis{\displaystyle}
\def\r{\rightarrow}
\def\ov{\overline}

\def\wh{\widehat}

\def\ve{\varepsilon}

\def\cB{{\cal B}}

\def\cL{{\cal L}}

\def\cH{{\cal H}}

\def\cZ{{\cal Z}}
\def\cR{{\cal R}}
\def\cO{{\cal O}}

\dimendef\dimen=0

\newtheorem{theorem}{Theorem}[section]
\newtheorem{lemma}[theorem]{Lemma}
\newtheorem{corollary}[theorem]{Corollary}
\newtheorem{proposition}[theorem]{Proposition}
\newtheorem{definition}[theorem]{Definition}
\newtheorem{remark}[theorem]{Remark}

\begin{document}

\begin{center}
{\LARGE \bf On Ulam Stability}
\end{center}

\begin{center}
M. Burger, N. Ozawa, A. Thom
\end{center}

\section{Introduction}

The topic of this paper is part of the general question, exposed for instance by S. Ulam in his book ``A collection of mathematical problems'' \cite{Ula60}, of the stability under ``quasification'' of certain objects in algebra, topology and analysis; see Chapter VI of \cite{Ula60}. More specifically define, according to Ulam, a $\delta$-homomorphism between groups $\Gamma, G$, where $G$ is equipped with a distance $d$, as a map
\begin{equation*}
\mu: \Gamma \rightarrow G
\end{equation*}
such that
\begin{equation*}
d(\mu(xy), \,\mu(x)\, \mu(y)) < \delta \;\; \mbox{for all} \;\; x,y \in \Gamma\,.
\end{equation*}

\medskip\n
The question is then in which situations such a map is ``close'' to an actual homomorphism. This problem has been considered and treated by Hyers and Ulam in the case where $\Gamma$ is a Banach space and $G = \IR$ the additive group of reals; when $\Gamma$ is an arbitrary group and $G = \IR$, a $\delta$-homomorphism is a quasi-homomorphism, a notion pertaining to the theory of bounded cohomology (see \cite{Mon06}, \cite{Ca} and references therein). 

\medskip
In this paper we are interested in the case where $G= U(\cH)$ is the group of unitary operators of a Hilbert space $\cH$ and the distance on $U(\cH)$ is simply $d(T,S) = \|T-S\|$ where $\| \; \|$ denotes the operator norm; in this context we will speak of unitary $\delta$-representations. We introduce the notion of a group $\Gamma$ being strongly Ulam stable as well as Ulam stable which loosely means, in the first case that any unitary $\delta$-representation is $F_\Gamma(\delta)$-near an actual unitary representation, where $\lim_{\delta \rightarrow 0} F_\Gamma(\delta) = 0$, while in the second case we only require this property for finite dimensional representations, in which case we will denote $F^{fd}_\Gamma$ the analogue of $F_\Gamma$; see Section 1 where we define these objects precisely. With this terminology, the result motivating the question addressed in this paper is

\begin{theorem}\label{theo0.1} {\rm (Kazhdan \cite{Kaz82})}

\medskip
Assume that $\Gamma$ is amenable. Then $\Gamma$ is strongly Ulam stable, in fact
\begin{equation*}
\dis\frac{\delta}{2} \le F_\Gamma(\delta) \le \delta + 120 \delta^2, \quad \forall \delta < \dis\frac{1}{10}\;.
\end{equation*}
\end{theorem}

This suggests the question of identifying natural classes of groups which are Ulam stable or strong Ulam stable. Concerning Ulam stability, D.~Kazhdan in the same article had given examples of $\frac{1}{n}$-representations of a compact surface group (of genus $\ge 2$) in $\IC^n$ which are $10^{-1}$-away from any unitary representation. Here we recall a construction of P.~Rolli giving in every finite dimension $n$, $\delta$-representations of the free group on two generators with dense image in $U(n)$ and $(2-\frac{\delta}{3})$-away from any unitary representation, $\forall \delta > 0$. Thus Ulam stability fails for non-abelian free groups. We show more generally that if the comparison map
\begin{equation}\label{*}\tag{*}
H^2_b(\Gamma, \IR) \rightarrow H^2(\Gamma, \IR)
\end{equation}

\n
is not injective, then $\Gamma$ is not Ulam stable, in fact we get the bound $F_\Gamma^{fd}(\delta) \ge \sqrt{3}$, $\forall \delta > 0$. Returning to strong Ulam stability, we show (Section 1) that if $\Gamma > \Lambda$, and $\Gamma$ is strongly Ulam stable, then $\Lambda$ is Ulam stable.  As a corollary to an effective version of this statement we obtain,

\begin{theorem}\label{theo0.2}
If $\Gamma$ contains a non-abelian free group, then $\Gamma$ is not strongly Ulam stable. In fact 
\begin{equation*}
F_\Gamma(\delta) \ge \frac{\sqrt{3}}{16}, \quad \forall \delta > 0\,.
\end{equation*}
\end{theorem}

In view of Kazhdan's theorem, it is thus natural to wonder whether strong Ulam stability characterizes amenability.

\medskip
The comparison map (\ref{*}) fails to be injective for instance if $\Gamma$ is non-elementary word hyperbolic, or a lattice in a connected simple Lie group of real rank $1$; in particular $\Gamma$ is then not Ulam stable. In contrast, when $\Gamma$ is an irreducible lattice in a connected semisimple Lie group of rank at least two, the comparison map (\ref{*}) is injective, in fact every $\IR$-valued quasihomomorphism is bounded,  (\cite{BuMo} Thm. 20,~21). Concerning the more general question of Ulam stability, which is still open, we have two partial results. The first one concerns lattices for which certain results about bounded generation of congruence subgroups are available, namely:

\begin{theorem}\label{theo0.3}
Let $\cO$ be the ring of integers of a number field, $S \subset \cO$ a multiplicatively closed subset and $\cO_S$ the corresponding localization. Then, for every $n \ge 3$, the group $SL(n, \cO_S)$ is Ulam stable.
\end{theorem}

The second result is a consequence of work of N.~Ozawa \cite{Ozawa} where he defines and studies property (TTT) for $SL(n,\IR)$ and its lattices; it leads to the following dimension dependent stability result:

\begin{theorem}\label{theo0.4}
Let $n \ge 3$ and $\Gamma < SL(n,\IR)$ be a lattice in $SL(n,\IR)$. Then, for any $d \in \IN$, there is a function $F_\Gamma^{(d)} : [0,\infty) \rightarrow [0,\infty)$ such that $\lim_{\delta \rightarrow 0^+} F_\Gamma^{(d)}(\delta) = 0$ and satisfies the following property: every $d$-dimensional unitary $\delta$-representation of $\Gamma$ is $F_\Gamma^{(d)}(\delta)$-close to a unitary representation.
\end{theorem}

When a $\delta$-representation is $F(\delta)$-close to a unitary representation, it is natural to ask whether the latter is unique. This leads to the notion of rigid and strongly rigid unitary representations which is a topic we study in Section 3. Again, those notions seem to draw a line between amenable and non-amenable groups, in that if $\Gamma$ is amenable, every unitary representation is strongly rigid, \cite{Joh86}, while we prove that if $\Gamma$ contains a non-abelian free group, the regular representations of $\Gamma$ on $\ell^2(\Gamma)$ is not strongly rigid.

\section{Ulam stability: definitions and lemmas}
\setcounter{equation}{0}

Let $\Gamma$ be a group, $\cH$ a Hilbert space; we let $\| \; \|$ denote throughout the paper, the operator norm on the ring $B(\cH)$ of bounded operators on $\cH$. We endow the space ${\rm Map}(\Gamma, U(\cH))$ of all maps of $\Gamma$ into the unitary group $U(\cH)$ of $\cH$ with the uniform distance,
\begin{equation*}
\| \mu - \nu\| : = \sup\limits_{\gamma \in \Gamma} \|\mu(\gamma) - \nu(\gamma)\|, \;\; \mu, \nu \in {\rm Map}(\Gamma, U(\cH))\,.
\end{equation*}

\medskip\n
Given $\mu \in {\rm Map} (\Gamma, U(\cH))$, we introduce now two important invariants,
\begin{align*}
D(\mu) &: = \inf\{\|\mu - \nu\|: \; \nu \in \; {\rm Hom} (\Gamma, U(\cH))\}
\intertext{and}
{\rm def}(\mu)& : = \sup\limits_{x,y \in \Gamma}\; \|\mu(xy) - \mu(x) \, \mu(y)\|
\end{align*}

\n
where the latter is referred to as the defect of $\mu$; notice that ${\rm def}(\mu) \le 3 D(\mu)$. Now for $\delta \ge 0$ we define:
\begin{equation*}
F_\Gamma(\delta) : = \sup\{D(\mu) : \;{\rm def}(\mu) \le \delta, \; \mu(e) = Id\}
\end{equation*}

\medskip\n
where the supremum is taken over all Hilbert spaces $\cH$ and maps $\mu$: $\Gamma \rightarrow U(\cH)$ with $\mu(e) = Id$. If instead we take the supremum only over all finite dimensional or $n$-dimensional Hilbert spaces, we will denote by $F^{fd}_\Gamma$, respectively $F_\Gamma^{(n)}$ the resulting function.

\medskip
We clearly have for all $\delta \ge 0$:
\begin{equation*}
F_\Gamma(\delta) \ge F_\Gamma^{fd}(\delta) \ge F_\Gamma^{(n)}(\delta) 
\end{equation*}
and
\begin{equation*}
F_\Gamma(0) = F_\Gamma^{fd}(0) = F_\Gamma^{(n)}(0) = 0\,.
\end{equation*}

\medskip\n
Observe in addition that $F_\Gamma^{(1)}(\delta) \ge \frac{\delta}{2}$; this is obtained by taking $\gamma_0 \in \Gamma$, $\gamma_0 \not= e$ and defining $\mu: \Gamma \r S^1$ with the property that $|\mu(\gamma_2) - 1| = \frac{\delta}{2}$ and $\mu(\gamma) = 1 \; \forall \gamma \not= \gamma_0$.

\medskip
Taking into account that the functions $F_\Gamma$, $F_\Gamma^{fd}$ and $F_\Gamma^{(n)}$ are all monotone increasing, we may define:
\begin{equation*}
{\rm def}(\Gamma) : = \lim\limits_{\delta \rightarrow 0^+} \, F_\Gamma(\delta)
\end{equation*}

\n
and analogously, ${\rm def}^{fd}(\Gamma)$ and ${\rm def}^{(n)}(\Gamma)$.

\begin{definition}\label{def1.1}
The group $\Gamma$ is strong Ulam stable, resp. Ulam stable, if ${\rm def}(\Gamma) = 0$, resp. ${\rm def}^{fd}(\Gamma) = 0$. Equivalently, the function $F_\Gamma$, resp. $F_\Gamma^{fd}$ is continuous at $\delta = 0$.
\end{definition}

We examine now the behaviour of these invariants under natural operations. In the sequel we will denote by ${\rm Map}_0(\Gamma, U(\cH))$ the set of maps $\mu: \Gamma \rightarrow U(\cH)$ with $\mu(e) = Id$.

\begin{lemma}\label{lem1.2}
Let $\pi : \Gamma \r \Gamma_0$ be a surjective homomorphism and $\mu \in {\rm Map}_0 (\Gamma_0, U(\cH))$.

\begin{itemize}
\item[{\rm (1)}]  ${\rm def}(\mu) = {\rm def}(\mu \circ \pi)$  

\item[{\rm (2)}] $\min (D(\mu \circ \pi),  \sqrt{3}) = \min(D(\mu), \sqrt{3})$.  
\end{itemize}
\end{lemma}

\begin{proof}
The first assertion as well as the inequality $D(\mu \circ \pi) \le D(\mu)$ are immediate. To show (2) we may thus assume that $D(\mu \circ \pi) < \sqrt{3}$ and pick $\varepsilon > 0$ and $\nu \in {\rm Hom}(\Gamma, U(\cH))$ with
\begin{equation*}
\|\mu \circ \pi(g) - \nu(g) \| \le \sqrt{3} - \varepsilon, \quad \forall g \in \Gamma\,.
\end{equation*}

\medskip\n
In particular, $\|\nu(g)^n - Id\| \le \sqrt{3} - \varepsilon$ for every $g \in {\rm Ker} \, \pi$, $n \in \IZ$. Thus, if $g \in {\rm Ker} \, \pi$, then for every $z \in {\rm spec}(\nu(g))$ and $n \in \IZ$, $|z^n - 1| \le \sqrt{3} - \varepsilon$ which implies that $z = 1$ and hence $\nu = \nu_0 \circ \pi$, for some $\nu_0 \in {\rm Hom}(\Gamma_0, U(\cH))$. Thus $D(\mu \circ \pi) = D(\mu)$ which concludes the proof.
\end{proof}

Lemma \ref{lem1.2} implies then immediately,

\begin{corollary}\label{cor1.3}
Let $\pi : \Gamma \r \Gamma_0$ be a surjective homomorphism. Then:
\begin{equation*}
\min({\rm def}(\Gamma_0), \sqrt{3}) \le \min ({\rm def}(\Gamma), \sqrt{3})
\end{equation*}

\medskip\n
and the same inequality holds with ${\rm def}$ replaced by ${\rm def}^{fd}$ or ${\rm def}^{(n)}$. In particular if $\Gamma$ is (strong) Ulam stable then $\Gamma_0$ is (strong) Ulam stable.
\end{corollary}

Now we turn to the problem of controlling the invariants introduced above for the operations of induction of maps, which we first have to define. Let $\Lambda < \Gamma$ be a subgroup, $\cR$ a set of representatives of the left $\Lambda$-cosets and:
\begin{align*}
r : &\;\; \Gamma \longrightarrow \cR
\\[1ex]
&\;\; \gamma \longmapsto \Lambda \gamma \cap \cR
\end{align*}
the corresponding retraction.

\medskip
Given a map $\mu: \Lambda \r U(\cH)$, we define the induced map $\ov{\mu}: \Gamma \r U(\ell^2(\cR, \cH))$ by 
\begin{equation*}
(\ov{\mu}(\gamma)  f) (x) : = \mu(x \gamma r(x\gamma)^{-1}) \,f(r(x\gamma)), \;\mbox{where} \; f \in \ell^2(\cR, \cH)\,.
\end{equation*}

\medskip\n
With these definitions and notations we have
\begin{lemma}\label{lem1.4} Let $\mu \in {\rm Map} (\Lambda, U(\cH))$ and $\ov{\mu} \in {\rm Map} \big(\Gamma, U(\ell^2(\cR, \cH))\big)$ the induced map,

\begin{itemize}
\item[{\rm (1)}] ${\rm def}(\ov{\mu})  = {\rm def}(\mu)$

\medskip
\item[{\rm (2)}] $\|\ov{\mu}_1 - \ov{\mu}_2\|  = \|\mu_1 - \mu_2\|$

\medskip
\item[{\rm (3)}] $D(\ov{\mu}) \le D(\mu)$.   
\end{itemize}
\end{lemma}

\medskip
\begin{proof}
(1): For $\gamma, \eta \in \Gamma$ and $x \in \cR$ we compute,
\begin{equation*}
\begin{array}{l}
(\ov{\mu}(\gamma \eta) - \ov{\mu}(\gamma) \ov{\mu}(\eta)) \, f(x) = 
\\[1ex]
\big(\mu(x \gamma \eta r(x \gamma \eta)^{-1}) - \mu (x \gamma r(x \gamma)^{-1}) \, \mu(r(x\gamma) \,\eta r(x \gamma \eta)^{-1})\big)\, \big(f(r (x \gamma \eta))\big)
\end{array}
\end{equation*}

\medskip\n
which implies, $\|\ov{\mu} (\gamma \eta) - \ov{\mu} (\gamma) \,\ov{\mu}(\eta)\| \le {\rm def} (\mu)$ and hence ${\rm def}(\ov{\mu}) \le {\rm def}(\mu)$.

\medskip
Setting $f = \delta_e \cdot \xi$, $\xi \in \cH$, taking $\gamma, \eta \in \Lambda$ and evaluating the above equality at $x = e$ gives
\begin{equation*}
(\ov{\mu}(\gamma \eta) - \ov{\mu}(\gamma) \,\ov{\mu}(\eta)) \,f(e) = (\mu(\gamma \eta) - \mu(\gamma) \,\mu(\eta)) (\xi)
\end{equation*}

\medskip\n
from which ${\rm def}(\ov{\mu}) \ge {\rm def}(\mu)$ follows.

\bigskip\n
(2): For $\mu_1,\mu_2: \Lambda \r U(\cH)$, $\gamma \in \Gamma$, $f \in \ell^2(\cR, \cH)$ and $x \in \cR$ we get
\begin{equation*}
(\ov{\mu}_1(\gamma) - \ov{\mu}_2(\gamma))\,f(x) = \big(\mu_1(x \gamma r(x \gamma)^{-1}) - \mu_2(x \gamma r(x \gamma)^{-1})) \,(f (r (x \gamma))\big)
\end{equation*}

\n
which first implies that
\begin{equation*}
\|\ov{\mu}_1 - \ov{\mu}_2\| \le \| \mu_1 - \mu_2\|\,.
\end{equation*}
Then let $\ve > 0$ and $\lambda \in \Lambda$ with
\begin{equation}\label{1.6}
\|\mu_1(\lambda) - \mu_2(\lambda)\| \ge \|\mu_1 - \mu_2\| - \ve\,.
\end{equation}

\n
Applying the equality above to $\gamma = \lambda$, $f = \delta_e \cdot \xi$, $\xi \in \cH$, we get:
\begin{equation*}
(\ov{\mu}_1(\lambda) - \ov{\mu}_2(\lambda)) \,f(e) = (\mu_1(\lambda) - \mu_2(\lambda))(\xi)
\end{equation*}

\medskip\n
from which, taking into account that $\|f\| = \|\xi\|$, follows:
\begin{equation*}
\|\ov{\mu}_1 (\lambda) - \ov{\mu}_2(\lambda) \| \ge \| \mu_1(\lambda) - \mu_2(\lambda)\|
\end{equation*}

\medskip\n
which together with (\ref{1.6}) concludes the proof. The inequality (3) follows from (2).
\end{proof}

The next task is to obtain a lower bound of $D(\ov{\mu})$ in terms of $D(\mu)$. This requires an additional hypothesis.

\begin{proposition}\label{prop1.5}
Then $\Lambda < \Gamma$, $\mu : \Lambda \r U(\cH)$ a map and $\ov{\mu}$: $\Gamma \r U(\ell^2(\cR, \cH))$ the induced map. Assume that ${\rm dim} \, \cH < + \infty$. Then,
\begin{equation*}
D(\mu) \le 16 \; D(\ov{\mu})\,.
\end{equation*}
\end{proposition}

We first need the following lemma which generalizes the well known fact that if a unitary representation has an almost invariant vector, then there is a nearby invariant vector.

\begin{lemma}\label{lem1.6}
Let $\nu$: $\Lambda \r U(\cL)$ be a unitary representation into a Hilbert space $\cL$ and assume that there is an orthogonal projection $P$ in $\cL$ such that

\begin{itemize}
\item[{\rm (1)}] $\|P - \nu(g) \,P \nu(g)^*\| < \delta \quad \forall g \in \Lambda$, 

\item[{\rm (2)}] the image of $P$ is finite dimensional.
\end{itemize}

\medskip\n
Then there exists an orthogonal projection $Q$ with

\begin{itemize}
\item[{\rm (1)}] $Q$ commutes with $\nu$,

\item[{\rm (2)}]  $\|P-Q\| \le 2 \delta$.
\end{itemize}
\end{lemma}

\begin{proof}
The basic observation is that $P$ belongs to the Hilbert spaces $HS(\cL)$ of Hilbert-Schmidt operators of $\cL$ on which $\Lambda$ acts unitarily by conjugation. If now $C$ is the convex hull of $\{\nu (g) \,P \nu(g)^*: g \in \Lambda\}$ in $HS(\cL)$, and $Q_0 \in \ov{C} \,^{HS}$ is the circumcenter of the closure of $C$ in the Hilbert-Schmidt norm, then $Q_0$ commutes with $\nu(G)$. Moreover, since $C \subset \{T: 0 \le T \le  Id$, $\|T - P\| \le \delta\}$ the same inclusion holds for $\ov{C}^{HS}$ and hence $0 \le Q_0 \le Id$, $\| Q_0 - P\| \le \delta$.

\medskip
From this we deduce that the spectrum of $Q_0$ is contained in $[0,\delta] \cup [1-\delta,1]$ and hence this holds for the spectrum of the support projection $Q$ of $Q_0$ as well; thus $\|Q - Q_0\| \le \delta$ and thus $\|P-Q\| \le 2 \delta$.
\end{proof}

\bigskip\n
{\it Proof of Proposition \ref{prop1.5}}~: Let $\mu: \Lambda \r U(\cH)$ be a map and $\ov{\mu}: \Gamma \r U(\ell^2(\cR,\cH))$ the induced map. We assume $D(\ov{\mu}) < \delta$ for some $\delta > 0$ and let $\nu: \Gamma \r U(\ell^2(\cR,\cH))$ be a unitary representation with $\|\nu(\gamma) - \ov{\mu}(\gamma)\| < \delta$, $\forall \gamma \in \Gamma$. Let $P$ be the orthogonal projection of $\ell^2(\cR,\cH)$ onto $\ell^2(\{e\}, \cH) = \cH$; then $P$ commutes with $\ov{\mu}(\lambda)$, $\forall \lambda \in \Lambda$ and hence $\|P- \nu(\lambda) \,P\nu(\lambda)^*\| < 2 \delta$, $\forall \lambda \in \Lambda$. Now apply Lemma \ref{lem1.6} to get the orthogonal projection $Q$ commuting with $\nu(\Lambda)$ and satisfying $\|P-Q\| \le 4 \delta$.

\medskip
Let $PQ = V\,|PQ|$ be the polar decomposition of $PQ$. Since $Q \ge QPQ \ge (1-4 \delta)Q$ one has $V^* V= Q$ and $\| Q- V\| \le \| \,|PQ| - Q\| \le 4 \delta$. It follows that $\|P-V\| \le \|P-PQ\| + \|PQ-V\| \le 8 \delta$. We note that $V^*V= Q$ commutes with $\nu(\Lambda)$ and $VV^* = P$. Therefore $V\nu(\cdot) V^*$ is a unitary representation of $\Lambda$ on $\ell^2(\{e\},\cH) = \cH$ and
\begin{equation*}
\|V \nu(g) \,V^* - \mu(g) \| \le 16 \delta \quad \forall g \in \Lambda \,.
\end{equation*}
\hfill $\square$

Using Prop.~\ref{prop1.5} and Assertion (1) in Lemma \ref{lem1.4}, we deduce
\begin{corollary}\label{cor1.7}
Let $\Lambda < \Gamma$. Then for every $\delta \ge 0$, we have
\begin{equation*}
F^{f d}_\Lambda(\delta) \le 16 \,F_\Gamma(\delta)\,.
\end{equation*}

\n
In particular, ${\rm def}^{fd}(\Lambda) \le 16\, {\rm def}(\Gamma)$ and if $\Gamma$ is strong Ulam stable then $\Lambda$ is Ulam stable.
\end{corollary}

\section{Fundamental examples of Ulam stable and non-Ulam stable groups}
\setcounter{equation}{2}

The following result, proved by D. Kazhdan, provides the only presently known examples of groups satisfying strong Ulam stability:

\begin{theorem}\label{theo2.1} {\rm (\cite{Kaz82})}

\medskip
Let $\Gamma$ be amenable. Then $\Gamma$ is strongly Ulam stable, in fact:
\begin{equation*}
\dis\frac{\delta}{2} \le F_\Gamma(\delta) \le \delta + 120 \delta^2, \quad \forall \delta \le \dis\frac{1}{10}\,.
\end{equation*}
\end{theorem}

\medskip\n
{\bf Question:} In the context of Theorem \ref{theo2.1}, does the limit $\lim\limits_{\delta \r 0^+} \;\frac{F_\Gamma(\delta)}{\delta}$ exist?

\medskip
For the convenience of the reader we include a short proof of Kazhdan's theorem in a slightly more general context; the basic idea is taken from \cite{Sh} Thm.~1.3:

\begin{theorem}\label{theo2.2}
Let $\Gamma$ be an amenable subgroup of $\Lambda, C \ge 1$, $0 < \varepsilon < (10C)^{-1}$, and $\pi: \Lambda \rightarrow B(\cH)$ be a map into the space $B(\cH)$ of bounded operators, such that for every $x \in \Gamma$ and $h \in \Lambda$, one has $\pi(e) = Id$, $\pi(x) \in U(\cH)$, $\|\pi(h)\| \le C$ and $\|\pi(xh) - \pi(x) \pi(h)\| \le \varepsilon$. Then, there is a map $\sigma$: $\Lambda \rightarrow B(\cH)$ such that for every $x \in \Gamma$ and $h \in \Lambda$, one has $\|\sigma(h) - \pi(h)\| \le \varepsilon + 120 C \varepsilon^2$, $\sigma(x) \in U(\cH)$, $\|\sigma(h)\| \le C$ and $\sigma (x h) = \sigma(x) \sigma(h)$.
\end{theorem}

\begin{proof}
We fix an invariant mean on $\Gamma$ and write it as $\int_\Gamma dx$. Define $\pi' : \Lambda \rightarrow B(\cH)$ by
\begin{equation*}
\pi '(h) = \dis\int_\Gamma \pi(x)^* \pi(x h) dx\,.
\end{equation*}

\n
One has $\pi'(e) = Id$ and $\|\pi(h) - \pi'(h)\| \le \varepsilon$ for all $h \in \Lambda$. Moreover, $\pi'(g)^* = \pi'(g^{-1})$ for $g \in \Gamma$. Let $g \in \Gamma$ and $h \in \Lambda$. Then,
\begin{equation*}
\begin{array}{l}
\dis\int_\Gamma (\pi(xg) - \pi(x) \pi(g))^* (\pi(xh) - \pi(x) \pi(h))dx
\\[1ex]
= \dis\int_\Gamma \pi(xg)^* \pi(xh) dx - \pi' (g)^* \pi(h) - \pi(g)^* \pi'(h) + \pi(g)^* \pi(h)
\\[2ex]
= \pi' (g^{-1} h) - \pi' (g)^* \pi' (h) + (\pi'(g) - \pi(g))^* (\pi'(h) - \pi(h))\,.
\end{array}
\end{equation*}
It follows that
\begin{equation*}
\|\pi'(g^{-1} h) - \pi' (g^{-1}) \pi'(h)\| = \|\pi' (g^{-1}h) - \pi' (g)^* \pi' (h)\| \le 2 \varepsilon^2\,.
\end{equation*}

\n
Since $Id -2 \varepsilon^2 \le \pi'(g)^* \pi'(g) \le Id$, one has $Id -2 \varepsilon^2 \le |\pi'(g)| \le Id$. Thus, for the unitary element $\pi_1(g) : = \pi'(g)|\pi'(g)|^{-1}$, one has $\|\pi_1(g) - \pi'(g)\| \le 3 \varepsilon^2$ (since $\varepsilon < 0.1)$. We set $\pi_1(h) = \pi'(h)$ for $h \in \Lambda \backslash \Gamma$. It follows that for every $x \in \Gamma$ and $h \in \Lambda$, one has $\|\pi(h) - \pi_1(h)\| \le \varepsilon + 3 \varepsilon^2$, $\pi_1(e) = Id$, $\pi_1(x) \in U(x)$, $\|\pi_1 (h)\| \le C$ and
\begin{equation*}
\|\pi_1(xh) - \pi_1(x) \pi_1(h)\| \le \|\pi'(xh) - \pi'(x)\pi'(h)\| + ( 6 + 3 C)\ve^2\,.
\end{equation*}

\n
Now, replacing $\pi$ with $\pi_1$ and $\varepsilon$ with $(6 + 3 C) \,\varepsilon^2$; and repeat the process. The sequence $(\pi_n(h))$ converges to $\sigma(h)$ and
\begin{equation*}
\|\pi(h) - \sigma(h)\| \le \varepsilon + (\sigma + 3 C) + (\sigma + 3 C)^3 \, \varepsilon^4 + \dots \le \varepsilon + 120 C \varepsilon^2\,.
\end{equation*}
\end{proof}

\medskip
In the same article, D. Kazhdan shows that if $\Gamma = \pi_1(S)$ is the fundamental group of a compact surface of genus at least $2$, then for every $n \ge 2$ there is a map $\mu_n$: $\Gamma \r U(n)$ such that
\begin{equation*}
{\rm def}(\mu_n) \le \dis\frac{1}{n} \;\; \mbox{and} \;\; D(\mu_n) \ge \dis\frac{1}{10}\;.
\end{equation*}

\medskip\n
In fact for free non-abelian groups there is a recent construction due to P. Rolli (\cite{5}, Prop. 5.1) giving ``$\varepsilon$-representations'' in every dimension with additional properties. Let us use 
\begin{equation*}
\cB_\delta = \{T \in U(n): \;\|T- Id\| \le \delta\}
\end{equation*}

\n
as notation for the $\delta$-ball around $Id$ in $U(n)$. We recall here P. Rolli's construction: let $\IF_2$ be the free group on generators $a,b$ and choose maps
\begin{equation*}
\tau_a, \tau_b: \; \IZ \longrightarrow B_{\frac{\delta}{3}} 
\end{equation*}
with
\begin{equation*}
\tau_a (k^{-1}) = \tau_a(k)^{-1}, \quad \tau_b (k^{-1}) = \tau_b(k)^{-1}, \quad \forall k \in \IZ\,.
\end{equation*}

\medskip\n
Define $\mu$: $\IF_2 \r U(n)$ on every reduced word
\begin{equation*}
w = a^{n_1} \,b^{m_1} \dots a^{n_k} \,b^{m_k}
\end{equation*}

\medskip\n
by $\mu(w) = \tau_a(n_1) \, \tau_b(m_1)  \dots  \tau_a (n_k) \,\tau_b(m_k)$. With these definitions we have,\begin{equation*}
{\rm def}(\mu) \le \delta\,.
\end{equation*}

\n
Let us now assume that for every $k \ge 1$ the product $\tau_a(k) \tau_b(k)$ has infinite order in $U(n)$. Then we claim that $D(\mu) \ge 2 - \frac{\delta}{3}$. Indeed, take $\ve > \frac{\delta}{3}$ and assume that there is $\nu \in {\rm Hom}(\Gamma, U(n))$ with
\begin{equation*}
\sup\limits_\gamma \;\|\mu(\gamma) - \nu(\gamma)\| \le 2 - \ve\,.
\end{equation*}

\n
In particular, 
\begin{align*}
\| \tau_a(k) - \nu(a)^k\| & \le 2-\ve 
\intertext{and}
\|\tau_b (k) - \nu(b)^k\| & \le 2- \ve, \quad \forall k \in \IZ
\end{align*}
which implies that
\begin{align*}
& \|\nu(a)^k - Id \| \le 2-\ve + \dis\frac{\delta}{3}, \quad \forall k \in \IZ
\\[1ex]
& \|\nu(b)^k - Id \| \le 2- \ve + \dis\frac{\delta}{3}\,.
\end{align*}

\n
Thus there exists an integer $m$ such that $\nu(a)^m = \nu(b)^m = Id$. As a result, ${\rm Ker} \, \nu \supset \langle a^m,b^m\rangle$ and thus:
\begin{equation*}
\|\mu((a^m\,b^m)^k) - Id\| \le 2- \ve , \quad \forall k \ge 1
\end{equation*}
and thus by the definition of $\mu$:
\begin{equation*}
\|(\tau_a(m) \,\tau_b(m))^k - Id \| \le 2-\ve
\end{equation*}

\n
and since $\tau_a(m)\, \tau_b(m)$ has infinite order, this is a contradiction. By choosing in addition the set
\begin{equation*}
S = \{\tau_a(m) \,\tau_b(m) : m \ge 1\}
\end{equation*}

\medskip\n
in such a way that $\bigcup_{k \ge 1} S^k$ is dense in $U(n)$ we obtain:

\begin{proposition}\label{prop2.2}
For every $\delta > 0$ there exists a map $\mu \in \IF_2 \r U(n)$ such that
\begin{itemize}
\item[{\rm (1)}] ${\rm def} \, \mu \le \delta \;\;\mbox{and} \;\; D(\mu) \ge 2 - \mbox{\footnotesize $\dis\frac{\delta}{3}$}$, in particular ${\rm def}^{(n)} (\IF_2) = 2$,

\item[{\rm (2)}] $\mu$ has dense image. 
\end{itemize}
\end{proposition}

\n
 Analogues with values in $\IR$ of maps with small defect are quasimorphisms. Recall that a function $\phi$: $\Gamma \r \IR$ is a quasimorphism if $d \phi(\gamma,\eta): = \phi(\gamma \eta) - \phi(\gamma) - \phi(\eta)$ is bounded on $\Gamma \times \Gamma$. It is homogeneous if $\phi(\gamma^n) = n \phi(\gamma)$, $\forall \gamma \in \Gamma$, $\forall n \in \IZ$. It is well known that every quasimorphism is at bounded distance from a homogeneous one, and, that if $QH_h(\Gamma, \IR)$ denotes the vector space of homogeneous quasimorphisms, the quotient $QH_h(\Gamma, \IR) / {\rm Hom}(\Gamma, \IR)$ describes the kernel of the comparison map: $H^2_b(\Gamma, \IR) \r H^2(\Gamma, \IR)$. The relation of quasimorphisms to our problem at hand is given by the following:

\begin{lemma}\label{lem2.3}
Let $\Gamma$ be a group and $\phi: \Gamma \r \IR$ a homogeneous quasimorphism which is not a homomorphism. Let 
\begin{equation*}
\mu = \exp ( 2 \pi \, i \, \phi)\,.
\end{equation*}
Then:
\begin{equation*}
\dis\frac{3}{\pi} \; \arcsin\;\dis\frac{D(\mu)}{2} + \|d \phi\|_\infty \ge 1\,.
\end{equation*}
\end{lemma}

\begin{proof}
Let $\ve > 0$ and $\nu \in {\rm Hom}(\Gamma, S^1)$ with 
\begin{equation*}
\sup\limits_\gamma \,|\mu(\gamma) - \nu(\gamma)| \le D(\mu) + \ve : = \delta\,.
\end{equation*}
Now write $\nu(\gamma) = \exp \; 2\pi \, i \, \varphi(\gamma)$ where $\varphi: \Gamma \r \IR$ satisfies $\| \phi - \varphi\|_\infty \le \frac{1}{2}$. We use the elementary geometric fact that
\begin{equation*}
\big| e^{2 \pi \, i \, x} - 1\big| \le \delta 
\end{equation*}
implies
\begin{equation*}
\dis\frac{2 \arcsin \; \dis\frac{\delta}{2}}{\delta} \ge \dis\frac{2 \pi \, x}{\big| e^{2 \pi \, i \, x}-1\big|} \ge 1 
\end{equation*}
to get,
\begin{align*}
2 \pi \, |\phi(\gamma) - \varphi(\gamma) | & \le \big| e^{2 \pi \, i(\phi (\gamma) - \varphi(\gamma))} - 1\big| \; \dis\frac{2 \arcsin \; \dis\frac{\delta}{2}}{\delta}
\\[2ex]
& \le 2 \arcsin \;\dis\frac{\delta}{2} , \quad \forall \gamma \in \Gamma
\end{align*}
from which follows:
\begin{equation*}
\|d \varphi\|_\infty \le \dis\frac{3}{\pi} \;\arcsin \; \dis\frac{\delta}{2} + \|d \phi \|_\infty\,.
\end{equation*}

\medskip\n
If now $\|d \varphi\|_\infty = 0$, then the homogeneous quasimorphism $\phi$, being at bounded distance from a homomorphism, would itself be a homomorphism. Thus $\|d \varphi\|_\infty > 0$; since $\nu$ is a homomorphism, $d \varphi$ takes values in $\IZ$ and thus $\|d \varphi\|_\infty \ge 1$.
\end{proof}

\begin{corollary}\label{cor2.4}
Assume that the comparison map $H^2_b(\Gamma, \IR) \r H^2(\Gamma, \IR)$ is not injective. Then $\Gamma$ is not Ulam stable, in fact
\begin{equation*}
{\rm def}^{(1)} (\Gamma) \ge \sqrt{3}\,.
\end{equation*}
\end{corollary}

\begin{proof}
Let $\phi$ be as in Lemma \ref{lem2.3} which we apply then to $t \cdot \phi$ and $\mu_t = \exp (2 \pi \, i \, t \, \phi)$, $t > 0$, to obtain:
\begin{equation*}
\dis\frac{3}{\pi} \; \arcsin \;\dis\frac{D(\mu_t)}{2} + t \cdot \| d \phi \|_\infty \ge 1\,,
\end{equation*}

\medskip\n
and letting $t \r 0$, we have ${\rm def} (\mu_t) \r 0$ while $\lim\limits_{t \r 0} \inf \,D(\mu_t) \ge \sqrt{3}$.
\end{proof}

\begin{remark}\label{rem2.5} \rm
The Corollary {\rm \ref{cor2.4}} applies to a large class of groups including non-elementary word hyperbolic groups and lattices in simple connected Lie groups of rank 1.
\end{remark}

Together with Corollary \ref{cor1.7} we deduce:
\begin{corollary}\label{cor2.6}
Assume that $\Gamma$ contains a subgroup $\Lambda$ such that:
\begin{equation*}
H^2_b (\Lambda, \IR) \r H^2(\Lambda, \IR)
\end{equation*}
is not injective. Then
\begin{equation*}
{\rm def}(\Gamma) \ge \dis\frac{\sqrt{3}}{16}
\end{equation*}

\n
in particular $\Gamma$ is not strong Ulam stable.
\end{corollary}

\begin{corollary}\label{cor2.7}
Assume that $\Gamma$ contains a non-abelian free group. Then $\Gamma$ is not strong Ulam stable.
\end{corollary}

\section{Deformation rigidity}
\setcounter{equation}{0}

If $\mu : \Gamma \rightarrow U(\cH)$ is a map with $D(\mu) < \varepsilon$, it is a natural question whether the representation $\omega: \Gamma \rightarrow U(\cH)$ with $\| \mu - \omega\| < \varepsilon$ is unique, provided that $\varepsilon$ is small. This question leads to the notion of rigid and strongly rigid representation which we introduce and study in this section.

\begin{definition}\label{def3.1}  ~

\begin{itemize}
\item[{\rm (1)}] $\pi \in {\rm Hom}(\Gamma, U(\cH))$ is rigid if there is $\ve > 0$, s.t.~whenever $\|\pi - \omega\| \le \ve$, where $\omega \in {\rm Hom}(\Gamma, U(\cH))$, then $\pi$ and $\omega$ are equivalent.

\item[{\rm (2)}] $\pi \in {\rm Hom}(\Gamma,U(\cH))$ is strongly rigid if $\pi$ is rigid and the orbit map
\begin{align*}
U(\cH)/ Z(\pi)  &\longrightarrow\; {\rm Hom}(\Gamma, U(\cH))\,,
\\[1ex]
U & \longmapsto \; U \,\pi \,U^{-1}
\end{align*}

\n
induces a homeomorphism onto its image. Here
\begin{equation*}
\cZ(\pi) = \{u \in U(\cH): u \pi(\gamma) u^{-1} = \pi(\gamma), \quad \forall \gamma \in \Gamma\}\,.
\end{equation*}
\end{itemize}
\end{definition}

\begin{theorem}\label{theo3.2} {\rm (\cite{Joh86})}

\medskip
If $\Gamma$ is amenable then every unitary representation is strongly rigid. 
\end{theorem}

Possibly more general, this result holds for unitarisable groups. See the work of G.\ Pisier \cite{Pisier}, for a general reference about unitarisable groups and related topics.

\begin{theorem}\label{theo3.3}
Assume that $\Gamma$ is unitarisable, that is, every uniformly bounded re\-presentation of $\Gamma$ on a Hilbert space is equivalent to a unitary one. Then every unitary representation is strongly rigid.
\end{theorem}

\begin{proof}
Assume that $\Gamma$ is not deformation rigid. Thus, for every $n$, there exist uniformly $1/n$-close unitary representations $\mu_n$ and $\nu_n$ of $\Gamma$ on a Hilbert space $\cH_n$ which are not conjugated by any unitary element $u$ such that $\|1 - u\| \le \ve$. We consider the representation $\pi_n = \mu_n \oplus \nu_n$ of $\Gamma$ on $\cH_n \oplus \cH_n$, and the derivation
\begin{equation*}
D_n(g) = [nE_{1,2}, \pi_n(g)] = \begin{pmatrix}
0 & n(\nu_n(g) - \mu_n(g)) 
\\ 
0 & 0 
\end{pmatrix} \in B(\cH_n \oplus \cH_n)\,.
\end{equation*}

\medskip\n
It is clear that $D_n$ is a derivation, $D_n(gh) = D_n(g) \,\pi_n(h) + \pi_n(g) D_n(h)$, such that $\sup_g \|D_n(g)\| \le 1$. Now, let $\pi = \bigoplus_n \,\pi_n$, $\cH = \bigoplus_n(\cH_n \oplus \cH_n)$ and $D = \bigoplus D_n$. Since $\Gamma$ is unitarisable, the uniformly bounded representation
\begin{equation*}
\pi_D(g) = \begin{pmatrix}
\pi(g) & D(g)
\\ 
0 & \pi(g)
\end{pmatrix} \in B(\cH \oplus \cH)
\end{equation*}

\n
is unitarisable, which implies that $D$ is inner. (For more details see \cite{MO}.) Hence, there exists $T \in B(\cH)$ such that $D(g) = [T, \pi(g)]$. Then, one has
\begin{equation*}
D_n(g) = [T_{nn}, \pi_n(g)] = T_{nn}^{1,2} \,\nu_n(g) - \mu_n(g) \,T_{nn}^{1,2}\,,
\end{equation*}

\n
where $T_{nn} \in B(\cH_n \oplus \cH_n)$ is the $(n,n)$-entry of $T \in B(\cH)$ and $T_{nn}^{1,2} \in B(\cH_n)$ is the $(1,2)$-entry of $T_{nn}$. It follows that $(n-T_{nn}^{1,2}) \nu_n(g) = \mu_n(g)$ $(n - T_{nn}^{1,2})$ for all $g$. But since $\sup_n \|T^{1,2}_{nn} \| \le \|T\|$, the operators $1-n^{-1} \,T^{1,2}_{nn}$ are invertible for large $n$, their unitary parts $u_n$ of the polar decomposition satisfy $u_n \nu_n(g) = \mu_n(g)u_n$ and $\|1-u_n\| \rightarrow 0$. This is a contradiction and finishes the proof.
\end{proof}

The rest of the section is devoted to the construction of examples $(\Gamma, \pi)$ where $\Gamma$ is a group and $\pi$: $\Gamma \r U(\cH)$ is non-rigid. First we establish a few straightforward facts:

\begin{proposition}\label{prop3.4}
Let $\Gamma > \Lambda$ be groups.

\bigskip\n
{\rm 1)} A finite dimensional unitary representation is strongly rigid. In fact, if $\| \pi - \omega\| \le \ve$ and $\ve \sqrt{n} < 1$, where $n = {\rm dim} \,\cH$, then $\pi$ and $\omega$ are equivalent via a unitary operator  $u \in U(\cH)$ with $\|u - Id\| \le \frac{3}{2} \, \ve \sqrt{n}$.

\bigskip\n
{\rm 2)} Let $\pi \in {\rm Hom} (\Gamma, U(\cH))$ and $\omega = {\rm Ind}_\Lambda^\Gamma \pi$ the induced representation. If $\omega$ is strongly rigid then $\pi$ is strongly rigid.

\bigskip\n
{\rm 3)} Assume $\Lambda \lhd \Gamma$ and let $p$: $\Gamma \r \Gamma / \Lambda$ be the canonical projection. If $\pi \circ p$ is (strongly) rigid then $\pi$ is (strongly) rigid.
\end{proposition}

The following will be useful:

\begin{lemma}\label{lem3.5}
Let $T \in B(\cH)$ with $\|T - Id\| \le \ve$, $\ve < 1$ and $T = U(T^* T)^{\frac{1}{2}}$ the polar decomposition. Then $\|U - Id\| \le 2 \ve$.
\end{lemma}

\begin{proof} We have,
\begin{equation*}
\| U - Id\|  \le \| T-Id\| + \|U(T^* T)^{\frac{1}{2}} - U\|  \le \ve + \|(T^* T)^{\frac{1}{2}} - Id\|\,.
\end{equation*}

\n
Now observe that for any unit vector $\xi$,
\begin{equation*}
\|(T^* T)^{\frac{1}{2}}\, \xi \| = \|T \xi \| \in [1-\ve, 1 + \ve]
\end{equation*}
and hence the spectrum of $(T^* T)^{\frac{1}{2}}$ is contained in $[1-\ve, 1 + \ve]$, which implies $\|T^* T)^{\frac{1}{2}} - Id\| \le \ve$, and concludes the proof.
\end{proof}

\medskip\n
{\it Proof of Proposition {\rm \ref{prop3.4}}:} 1) Setting $\ell(\gamma) T : = \pi (\gamma) \,T \omega(\gamma)^{-1}$ for $T \in B(\cH)$ and $\gamma \in \Gamma$, we obtain a unitary representation of $\Gamma$ into the space $B(\cH)$ endowed with the Hilbert-Schmidt norm. For $T = I$ we get:
\begin{align*}
\| \ell(\gamma) \,I - I\|_{HS} & = \| \pi(\gamma) - \omega(\gamma)\|_{HS}
\\[1ex]
& \le  \sqrt{n} \,\|\pi(\gamma) - \omega(\gamma)\| \le \sqrt{n} \cdot \ve\,.
\end{align*}

\n
Since $\|I\|_{HS} = \sqrt{n}$ and $\ve < 1$, there is a $\ell(\Gamma)$-invariant vector $T$ in the closed convex hull of $\{ \ell(\gamma)\,I: \, \gamma \in \Gamma\}$; then $\|T-I\|_{HS} \le \sqrt{n} \, \ve < 1$ and hence $T$ is an invertible operator with $\omega(\gamma) = T^{-1} \pi(\gamma) T$, $\forall \gamma \in \Gamma$; applying Lemma \ref{lem3.5} we get $u \in U(\cH)$ with $\omega(\gamma) = u^{-1} \pi(\gamma) u$ and $\|u - Id\| \le 3 \sqrt{n} \,\ve$.

\bigskip\n
2) Let $\pi \in {\rm Hom}(\Lambda, U(\cH))$, $\omega = {\rm Ind}_{\Lambda}^{\Gamma} \pi$ the induced representation into the Hilbert space $\cL = \ell^2(\cR, \cH)$, where $\Gamma = \bigcup_{\gamma \in \cR} \gamma \Lambda$. Let $\ve > 0$ and $\delta (\ve) > 0$ such that whenever $\|\omega' - \omega\| \le \delta (\ve)$, there is $u \in U(\cL)$ with: $\omega' = u^{-1} \omega u$ and $\|u - I\| \le \ve$.  Let $\pi' \in {\rm Hom}(\Gamma, U(\cH))$ with $\|\pi - \pi'\| \le \delta (\ve)$ and $\omega' = {\rm Ind}^\Gamma_\Lambda \pi'$. Then (Lemma \ref{lem1.4} (2)), $\|\omega-\omega'\| \le \delta(\ve)$ and thus there is $u \in U(\cL)$ with $\|u - I\| \le \ve$ and $\omega' = u^{-1} \omega u^{-1}$. Let $P$ be the orthogonal projection of $\cL = \cL(\cR, \cH)$ onto $\cL(e,\cH) = \cH$. Then $T = P \,u|_{\cH} : \cH \r \cH$ intertwines $(\omega' |_\Lambda)|_\cH$ and $(\omega|_\Lambda)|_\cH$,  that is $\pi'$ and $\pi$, and $\| T - Id\| \le \ve$ which allows to conclude using Lemma \ref{lem3.5}.

\bigskip\n
3) Straightforward. \hfill $\square$

\bigskip
Now we turn to a basic example of non-rigidity which concerns the free group $\IF_{\infty}$ on countably many generators and uses essentially a construction of Pytlik and Szwarc, see \cite{PS} 2.2 and 2.3,

\begin{theorem}\label{theo3.6}
The left regular representation $\lambda$ of $\IF_{\infty}$ on $\ell^2(\IF_{\infty})$ is not rigid. In fact there is a path of representations:
\begin{align*}
[0,1] & \longrightarrow \; {\rm Hom}(\IF_{\infty}, U \ell^2(\IF_\infty))
\\[1ex]
r & \longmapsto \; \pi_r
\end{align*}
such that 

\medskip\n
{\rm (1)} $\pi_0 = \lambda$.
 
\medskip\n
{\rm (2)}  $\pi_r$ and $\pi_{r'}$ are irreducible, inequivalent $\forall r > r' > 0$.

\medskip\n
{\rm (3)} $r \longmapsto \pi_r$ is continuous in the uniform topology of ${\rm Hom}(\IF_{\infty}, U \ell^2(\IF_\infty))$.
\end{theorem}

Using Proposition \ref{prop3.4} (2) we deduce

\begin{corollary}\label{cor3.7}
Assume that $\Gamma$ contains a non-abelian free group. Then the regular representation of $\Gamma$ in $\ell^2(\Gamma)$ is not strongly rigid.
\end{corollary}

\n
{\it Proof of Theorem {\rm \ref{theo3.6}}:} Set $F = \IF_\infty$ for simplicity. For $a \in F$ define $\ov{a}$ to be the unique word which is obtained from $a$ by deleting the last letter. On the Hilbert space $\ell^2 F$, consider the operator $P$ which is given by the formula
\begin{equation*}
P(\delta_a) = \delta_{\ov{a}} \;\mbox{for $a \not= e$ and} \; P(\delta_e) = 0\,.
\end{equation*}

\n
For $a \in F$, consider the finite dimensional subspace $K(a)$, which arises as the linear span of $\{\delta_a,\delta_{\ov{a}}, \dots, \delta_e\}$. Then we have,

\bigskip\n
(1) The image of $P-\lambda(a)\,P\lambda(a)^{-1}$ is contained in $K(a)$.

\medskip\n
(2) $\|P-\lambda(a)\,P\lambda(a)^{-1}\| \le 2$.

\medskip\n
(3) The operator $P$ preserves $K(a)$ and its restriction to $K(a)$ is a contraction.

\bigskip
Since $P$ is locally nilpotent on the linear span of $\{\delta_a \in a \in F\}$, the formula
\begin{equation*}
\pi^o_z(a) = (1-zP)^{-1} \,\lambda(a)(1-zP)
\end{equation*}

\medskip\n
defines a representation of $F$ on the linear span of $\{\delta_a, a \in F\}$. Now define $T_z = Id + (\sqrt{1 - z^2} - 1)T$, where $T$ is the orthogonal projection onto $\IC \delta_e$ and one has chosen the principal branch of the square root, finally set
\begin{equation*}
\pi_z(a) = T_z \pi_z^\circ (a)\,T_z^{-1}, \quad \forall a \in F\,.
\end{equation*}

\n
Assertions (1) and (2) are proven in \cite{PS}, Theorem 1. In order to see that for $z$ with $|z| < 1$, we have
\begin{equation*}
\sup\limits_{a \in F} \, \|\pi_z(a) - \pi_w(a)\| \r 0, \quad {\rm as} \; w \r z
\end{equation*}

\n
it is sufficient to show the corresponding estimate for $\pi^\circ_w(a)$. We compute
\begin{equation*}
\pi^o_z(a) \lambda(a)^{-1} = (1 - z P)^{-1} \,\lambda(a)(1-zP) \,\lambda(a)^{-1} = 1 + \Big( \dis\sum\limits_{n=0}^\infty \,z^n \,P^n\Big)\; (P-\lambda(a) \,P\lambda (a)^{-1})
\end{equation*}
and can see that
\begin{equation*}
\pi^o_z(a) \lambda(a)^{-1} - \pi^o_w(a) \lambda(a)^{-1} = \Big( \dis\sum\limits_{n=0}^\infty \,(z^n- w^n)P^n\Big)\; (P-\lambda(a) \,P\lambda (a)^{-1})\,.
\end{equation*}

\n
For $\xi \in \ell^2(F)$, the vector $(P-\lambda(a)\,P\lambda(a)^{-1}) \xi$ lies in $K(a)$. Since $P$ preserves $K(a)$ and is a contraction on $K(a)$, this implies
\begin{align*}
\|\pi^o_z(a) \lambda(a)^{-1} \,\xi - \pi^o_w(a) \lambda(a)^{-1} \, \xi\| & \le \dis\sum\limits^\infty_{n=0} \,|z^n - w^n| \, \|\xi \|, \quad \forall \xi \in \ell^2(F)\,.
\intertext{Hence}
\|\pi^o_z(a)- \pi^o_w(a)\| & \le \dis\sum\limits^\infty_{n=1} \,|z^n - w^n|\,.
\end{align*}

\medskip\n
This implies the claim since, for $z$ with $|z| < 1$, we have $\sum_{n=1}^\infty |z^n - w^n| \r 0$ as $w \r z$. \hfill $\square$

\section{Stability for special linear groups of integral matrices}
\setcounter{equation}{0}

In this section we establish the two stability results announced in the introduction.

\medskip
Let $K$ be a finite extension of $\IQ$, $\cO \subset K$ the ring of algebraic integers and $B \subset \cO$ an order in $\cO$. Let $S \subset B$ be a multiplicatively closed subset and denote by $A = B_S$ the localization of $B$ at $S$. Throughout this section the ring $A$ will be fixed.

\begin{theorem}\label{theo4.1}
If $n \ge 3$, then $SL(n, A)$ is Ulam stable.~In fact there exists  \linebreak $c = (n,A) > 0$ such that
\begin{equation*}
\dis\frac{\delta}{2} \le F_{SL(n,A)}^{fd} (\delta) \le c \delta \; \mbox{for all $\delta \ge 0$}\,.
\end{equation*}
\end{theorem}

\medskip
Let $n \in \IN$ and $q \subset A$ be an ideal. We denote by $E^\bullet (n,A; q)$ the set of elementary matrices in $SL(n,A)$ which are congruent to $1_n$ modulo $q$. Let $E^\bullet(n,A;q)$ denote the closure of $E^\bullet (n,A;q)$ inside $SL(n,A)$ under conjugation. We denote by $E(n, A;q)$ the subgroup of $SL(n,A)$ generated by $E^\bullet(n,A;q)$; finally denote by $SL(n,A;q)$ the congruence subgroup formed by those elements which are congruent to the identity modulo $q$.

\begin{lemma}\label{lem4.2}
The subgroup $E(n,A;q) \subset SL(n,A)$ is normal and of finite index.
\end{lemma}

\begin{theorem}\label{theo4.3} {\rm (\cite{W} (Cor. 3.1.3)} ~

\medskip
Let $A$ be a ring as above, let $n \ge 3$ be an integer and $q \subset A$ be an ideal in $A$. There exists $r(A,n) \in \IN$, not depending on the ideal $q$, such that the set $E^\bullet(n,A;q), r(A,n)$-boundedly generates, that is, every element in $E(n,A;q)$ is a product of at most $r(A,n)$ elements of $E^\bullet(n,A;q)$.
\end{theorem}

\begin{lemma}\label{lem4.4}
Let $A$ be a ring as above. Every $k$-dimensional unitary representation of $A^2$ whose character is $SL(2,A)$-invariant factors through $(A/qA)^2$ for some $q \in \IN$.
\end{lemma}

\begin{proof}
For $R = \IZ$, this is a (more or less) classical fact. Indeed, the character of the representation defines a $SL(2, \IZ)$-invariant probability measure on $\IT^2$. At the other side this measure is also a finite sum of atoms $\phi_1,\dots,\phi_k$. In particular, for each of the atomes $\phi_1$, there has to exist $n_i \in \IN$ such that the matrix
\begin{equation*}
\begin{pmatrix}
1 & n_i 
\\
0 & 1
\end{pmatrix} \in SL(2,\IZ)
\end{equation*}
fixes $\phi_i$. For $q = n_1 \dots n_k$, the matrix
\begin{equation*}
\begin{pmatrix}
1 & q 
\\
0 & 1
\end{pmatrix} \in SL(2,\IZ)
\end{equation*}

\medskip\n
will fix each of the characters. This just means that $\phi_i(t,s) = \phi_i(t + qs,s)$, for $1 \le i \le k$, and hence $\phi_i(qs,0) = 1$ for all $s \in \IZ$. Taking everything together, we conclude that $(q\IZ,0)$ acts trivially. The claim now follows since $(0,q\IZ)$ is in the orbit of $(q\IZ,0)$ under the $SL(2,\IZ)$-action.

\medskip
For more general $A$, we can conclude from the case $A = \IZ$ that for some $q \in \IZ \subset A$, the elements $(q,0)$ and $(0,q)$ act trivially. By $SL(2,A)$-covariance, we get that $g.(q,0)$ acts trivially for every $g \in SL(2,A)$. Let $a \in A$ be arbitrary and $b \in A$ such that $ab = n\in \IZ$. Then
\begin{equation*}
g = \begin{pmatrix}
a & n-1 
\\
1 &b
\end{pmatrix} \in SL(2,A)\,,
\end{equation*}

\medskip\n
and $g.(q,0) = (aq,q)$. We conclude that $(aq,0)$ and hence the whole ideal $qA \subset A$ which is generated by $q$ acts trivially. This finishes the proof.
\end{proof}

\medskip\n
{\it Proof of Theorem {\rm \ref{theo4.1}}:} Let $\pi: SL(n,A) \r U(k)$ be a finite dimensional unitary $\ve$-representation. By Kazhdan's Theorem (see Thm. 0.1), the restriction of $\pi$ to a standard unipotent copy of $A^2$ is $2\ve$-close to a unitary representation $\mu$. For $g \in SL(2,A)$, the representation $A^2 \ni t \longmapsto \mu(g.t)$ is $7\ve$-close to $t \longmapsto \pi(g)\,\mu(t) \,\pi(g)^{-1}$. Indeed, we compute
\begin{align*}
\mu(g.t)& \sim _{2 \ve} \,\pi(g.t) = \pi(gtg^{-1}) \sim
\\[1ex]
&\sim _{2 \ve} \,\pi(g) \, \pi(t) \, \pi(g^{-1})  \sim _{2 \ve}\,\pi(g) \, \mu(t) \, \pi(g^{-1})  \sim _{\ve}\,\pi(g) \, \mu(t) \, \pi(g)^{-1} \,,
\end{align*}

\n
where here and in the sequel we use the symbol $\sim_\varepsilon$ to indicate that two maps are $\varepsilon$-close.

\medskip
For $\ve < \frac{1}{7}$, this implies by Johnson's theorem (Thm. 3.2) that $t \longmapsto \mu(g.t)$ is unitarily equivalent to the unitary representation $\mu$. We conclude that the character of $\mu$ (viewed as an atomic measure on $\wh{A}^2$) is fixed by the natural action of $SL(2,A)$. This implies by Lemma \ref{lem4.4} that there exists $0 \not= q \in \IZ \subset A$ such that $\mu(a,0) = 1_n$ for all $a \in qA$.

\medskip
We conclude that $\pi(b,0)  \sim _{2 \ve}\, I_{d_k}$ for $b \in qA$. Since all elements in $E^\bullet(n,A;qA)$ are conjugate to $(b,0) \in A^2$ for some $b \in qA$, we conclude that for $g \in E^\bullet(n,A; qA)$, 
\begin{align*}
\pi(g) = \pi(t(b,0)t^{-1})  \sim _{2 \ve}\,\pi(t) \,\pi(b,0) \,\pi(t^{-1})  \sim _{2 \ve} \,\pi(t)\, \pi(t^{-1})  \sim _{2 \ve} \, I_{d_k}\,.
\end{align*}
Hence
\begin{equation*}
\pi(g)  \sim _{5 \ve}  \,I_{d_k}, \quad \forall g \in E^\bullet(n,A;qA)\,.
\end{equation*}

\medskip\n
Now, since $E(n,A; qA)$ is $r(A,n)$-boundedly generated by $E^\bullet(n,A;qA)$ (Thm. 4.3) we see that
\begin{equation*}
\pi(g)  \sim _{5 \ve (A,n)\ve} \,I_{d_k}, \quad \forall g \in E(n,A; qA)\,.
\end{equation*}

\n
Consider now the finite quotient
\begin{equation*}
Q(n,A;qA) = \dis\frac{SL(n,A)}{E(n,A;qA)}
\end{equation*}

\medskip\n
and pick a section $\sigma: Q(n,A; qA) \r SL(n,A)$. The composition $\pi \circ \sigma$: $Q(n,A;qA) \r U(k)$ defines a $(5 r(A,n) + 2)\ve$-representation of $Q(n,A;qA)$. Indeed, for $g,h \in Q(n,A; qA)$ we can compute
\begin{equation*}
\begin{array}{l}
\pi(\sigma(gh) = \pi(\sigma(g) \,\sigma(h) \,\sigma(h)^{-1} \,\sigma(g)^{-1} \,\sigma(gh))
\\[1ex]
\sim _{2 \ve} \,\pi(\sigma(g)) \,\pi(\sigma(h)) \,\pi(\sigma(h)^{-1} \,\sigma(g)^{-1}\,\sigma(gh)) \,\sim _{5 r(A,n)\ve} \,\pi(\sigma(g)) \,\pi(\sigma(h))\,.
\end{array}
\end{equation*}

\medskip\n
By Kazhdan's Theorem, $\pi \circ \sigma$ is $(10r(R,n) + 4)\ve$-close to a unitary representation $\pi''$ of $Q(n,A;qA)$. Composing $\pi''$ with the natural quotient map from $p: SL(n,A) \r Q(n,A;qA)$ we obtain a unitary representation $\pi' = \pi'' \circ p$ of $SL(n,A)$. Moreover, for all $g \in SL(n,A)$
\begin{equation*}
\begin{array}{l}
\pi'(g) = \pi'' \circ p(g) \sim _{(10 r(A,n) + 4)\ve} \,\pi \circ \sigma \circ p(g) \sim 
\\[1ex]
\sim_\ve \,\pi(g) \,\pi(g^{-1} \,\sigma(p(g)) \sim_{(5r(A,n) + 2)\ve} \,\pi(g)\,,
\end{array}
\end{equation*}

\medskip\n
and hence $\pi$ is $(15 r(A,n) + 7)\ve$-close to the unitary representation $\pi'$. This finishes the proof. \hfill $\square$

\medskip
Finally we turn to Theorem \ref{theo0.4} in the introduction; this will actually follow from Ozawa's theorem (\cite{Ozawa} Thm.~B) that any lattice in $SL(n,\IR)$, $n \ge 3$, satisfies property (TTT) and the following general result:
\begin{theorem}\label{theo4.5}
Let $\Gamma$ be a group which satisfies property {\rm (TTT)}. Then for every $d \in \IN$, 
\begin{equation*}
\lim\limits_{\delta \to 0^+} \,F_\Gamma^{(d)} (\delta) = 0\,.
\end{equation*}
\end{theorem}

\begin{proof}
We prove by contradiction, and assume  that there are $d \in \IN$, $\varepsilon > 0$ and $(1/n)$-representations $\pi_n: \Gamma \rightarrow U(d)$ which are $\varepsilon$ away from honest representations. We may moreover assume that $\pi_n$ converges pointwise to a map $\pi$, which is necessarily a representation. We view the $d \times d$ matrix algebra as the Hilbert space $H S(d)$ with the normalized Hilbert-Schmidt norm (it is normalized so that $\|I\|_{HS} = 1$), and consider $\sigma_n : \Gamma \rightarrow B(H S(d))$ defined by
\begin{equation*}
\sigma_n(g)x = \pi_n(g) x\pi(g)^*\,.
\end{equation*}
Then, one has
\begin{equation*}
\|\sigma_n(gh) I - \sigma_n(g) \sigma_n(h) I\|_{HS} = \|\pi_n(gh) - \pi_n(g) \pi_n(h) \|_{HS} \le \dis\frac{1}{n} \rightarrow 0
\end{equation*}
uniformly for $g, h \in \Gamma$, and
\begin{equation*}
\|\sigma_n(g) I - I\|_{HS} = \|\pi_n(g) - \pi(g)\|_{HS} \rightarrow 0
\end{equation*}

\n
for every $g \in \Gamma$. It follows from property (TTT) (or property $(T_Q)$, see \cite{Ozawa} Sect. 3) that
\begin{equation*}
\|\pi_n(g) - \pi(g)\| \le d^{1/2} \|\pi_n(g) - \pi(g)\|_{HS} = d^{1/2} \|\sigma_n(g) I - I\|_{HS} \rightarrow 0
\end{equation*}
uniformly for $g \in \Gamma$. A contradiction.
\end{proof}

\vfill\n
{\it Affiliation:}

\medskip\n
M. Burger, Departement Mathematik, ETH-Z\"urich, 8092 Z\"urich, Switzerland. \\
burger@ math.ethz.ch

\bigskip\n
N. Ozawa, Department of Mathematical Sciences, Tokyo, 153-8914, Japan. \\
narutaka@ms.u-tokyo.ac.jp

\bigskip\n
A. Thom, Mathematisches Institut, Universit\"at Leipzig, Johannisgasse 26, 04103 Leipzig, Germany.\\
thom@math.uni-leipzig.de

\end{document}